\documentclass[11pt]{article}
\pdfoutput=1
\usepackage[ascii]{inputenc}
\usepackage[T1]{fontenc}
\usepackage[english]{babel}
\usepackage{amsmath,amssymb,amsfonts,textcomp,epsfig,amsthm}
\usepackage[pdftex,colorlinks=true,citecolor=blue]{hyperref}
\usepackage{color}
\usepackage{calc}
\usepackage[all]{xy}
\usepackage{graphicx}
\usepackage{tikz}

  \newcommand{\R}{\mathbb{R}}
\newcommand{\C}{\mathbb{C}}

\newtheorem{thm}{Theorem}
\newtheorem{rmk}{Remark}
\newtheorem{coro}{Corollary}

\newtheorem{lem}{Lemma}

\title{Perturbations of symmetric elliptic Hamiltonians of degree four in a complex domain}
\author{
Bassem Ben Hamed\\
Institut Sup\'erieur d'Electronique et de Communication de Sfax, \\Route Menzel Chaker km 0.5, BP 261, 3038 Sfax, Tunisie.\\
E-mail: \texttt{bassem.benhamed@gmail.com}\\
\\
Ameni Gargouri\\
Facult\'e des Sciences de Sfax, D\'epartement de Math\'ematiques,\\
BP 1171, 3000 Sfax, Tunisie.\\
E-mail: \texttt{ameni.gargouri@gmail.com}\\
\\
Lubomir Gavrilov\\
Institut de Math\'{e}matiques de Toulouse, UMR 5219\\
Universit\'{e}  de Toulouse,  31062 Toulouse,  France.\\
E-mail: \texttt{lubomir.gavrilov@math.univ-toulouse.fr}}

\date{}
\begin{document}
\maketitle
\date
\begin{abstract}
The cyclicity of the exterior period annulus of  the asymmetrically perturbed Duffing oscillator is a well known problem extensively studied in the literature. In the present paper we provide a complete bifurcation diagram for the number of the zeros of the associated Melnikov function in a suitable complex domain. 
\end{abstract}

\section{Introduction}
\label{section0}

\noindent

Consider the asymmetrically perturbed Duffing oscillator
\vskip0.2cm
\begin{eqnarray}\label{f}
X_{\lambda}:
\left\{\begin{array}{ccl} \dot{x}&=&y  \\
\dot{y}&=&x-x^{3}+\nu x^2+ \lambda_{0}y+\lambda_{1}xy+\lambda_{2}x^{2}y
\end{array}\right.
\end{eqnarray}
in which $\nu, \lambda_i$ are small real parameters. For $\nu=\lambda_1=\lambda_2=\lambda_3=0$ the system
is integrable, with a first integral
$$
H= \frac{y^2}{2} - \frac{x^2}{2} + \frac{x^4}{4}
$$
 and its phase portrait is shown on fig.\ref{fig1}. Alternatively, the system (\ref{f}) defines a real plane foliation by the formula
\begin{equation}
 d(H - \nu \frac{x^3}{3})+ ( \lambda_{0}+\lambda_{1}x+\lambda_{2}x^{2}) ydx = 0
 \label{ff}
\end{equation}
 \begin{figure}
\begin{center}
 \def\svgwidth{9cm}
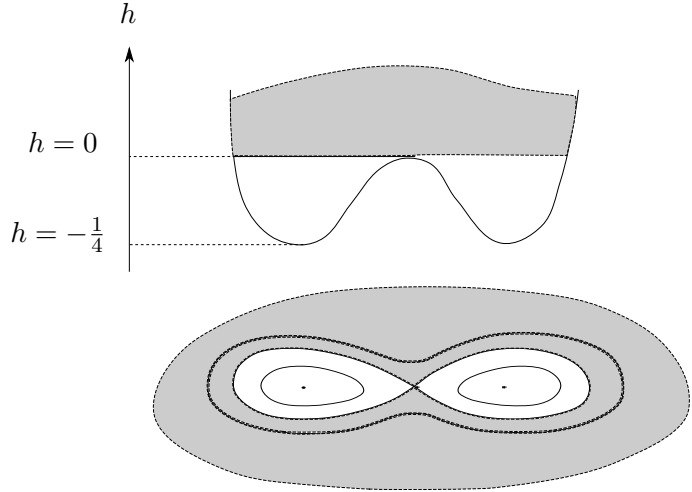
\end{center}
\caption{Phase portrait of $X_0$ and the graph of  $- \frac{x^2}{2} + \frac{x^4}{4}$}
\label{fig1}
\end{figure}

The cyclicity of the exterior period annulus of this system with respect to the perturbation (\ref{f}) is equal to two, as it has been shown
by Iliev and Perko \cite{Iliev99} and Li, Mardesic and Roussarie \cite{Li06}.
\begin{thm}
\label{cyclicity}
The cyclicity of the exterior period annulus $\{(x,y)\in \R^2: H(x,y)>0 \}$ of $dH=0$ with respect to the perturbation (\ref{f}) equals two.
\end{thm}
\begin{rmk}
\label{rem1}
The above Theorem claims that from any compact, contained in the open exterior period annulus $\{(x,y)\in \R^2: H(x,y)>0 \}$, bifurcate at most two limit cycles. It says nothing about the limit cycles bifurcating from the separatrix eight-loop or from infinity (i.e. the equator of the Poincar\'e sphere).
\end{rmk}

Let $\{\gamma(h)\}_h$ be the continuous family of exterior ovals of the non-perturbed system, where  $$\gamma(h)~\subset\{H=h\}$$ and consider the complete elliptic integrals
\noindent
\begin{eqnarray}\label{3}
I_{i}&=&\oint_{\gamma(h)} x^{i}y dx.
\end{eqnarray}
It has been shown in \cite{Iliev99}, that if we restrict our attention to a one parameter deformation
$$
\lambda_i=\lambda_i(\varepsilon), \nu= \nu (\varepsilon)
$$
then the first non-vanishing Poincar\'e-Pontryagin-Melnikov function $M_k $ (governing the bifurcation of limit cycles) is given by a linear combination of the complete elliptic integrals of first and second kind $I_0,I_2, I_4'$
\begin{equation}
\label{mk}
M_k(h)= \lambda_{0k}I_0(h)+ \lambda_{2k}I_2(h) + \lambda_{4k}I_4'(h) .
\end{equation}
It is shown further, by making use of Picard-Fuchs equations combined with Rolle's theorem in a real domain, that the space of elliptic integrals of first and second kind $I_0,I_1,I_4'$ satisfy the Chebishev property. The method is described in details in \cite[Iliev]{Iliev98} .
This result is further generalized for multi-parameter deformations. It turns out that the Bautin ideal associated to the deformation can be always principalized (this is a general fact), and that the leading term of the displacement map is given by a function of the form (\ref{mk}), which completes the proof of Theorem \ref{cyclicity}, see  \cite{Li06} for details.

The purpose of the present paper is to study the number of the zeros 
of
 the family $\{I_0,I_1,I_4'\}$ in the   complex domain ${\cal D} = {\mathbb C} \setminus (-\infty,0]$. We use the well known Petrov method which is based on the argument principle. To find the exact number of zeros we construct the bifurcation diagram of zeros of $M_k$ in ${\cal D}$  in the spirit of \cite[fig.4]{Gav01}. The result is summarized in Theorem \ref{main}. This gives an information on the complex limit cycles of the system, and imples in particular that the number of corresponding limit cycles can not exceed three. It can be also seen as a complex counterpart of Theorem \ref{cyclicity}.

Our primary motivation was that the complex methods we use, are necessary to understand the bifurcations from the separatrix eight-loop, see Remark \ref{rem1} above.  Another reason is, that the complexity of the bifurcation set of $M_k$ in a complex domain is directly related to the number of the zeros of $M_k$. This observation can be possibly  generalized to higher genus curves.

The paper is organized as follows. In section \ref{section1} we recall some known  Picard-Fuchs equations, which will be used later. The monodromy of the Abelian integrals, based on the classical Picard-Lefschetz theory is described in section \ref{monodromy}.
The Petrov method is then applied in section \ref{zeros}. The main result is that the principal part of the first rerun map can have at most four zeros in a complex domain, a result which is not optimal - see Lemma \ref{z3}. The exact upper bound for the number of the zeros in a complex domain turns out to be three. This result, together with the bifurcation diagram of zeros in a complex domain is given in section \ref{section5}.

\section{Picards-Fuchs equations }
\label{section1}
The results of this section are known, or can be easily deduced, see \cite{Iliev99, jezo94, Zol91}.

First we note that the affine algebraic curve
$$
\Gamma_h= \{(x,y)\in \C: H(x,y)=h \}
$$
is smooth for $h \neq 0,-1/4$ and has the topological type of a torus with two removed points $\infty^\pm$ (at "infinity"). Its homology group is therefore of rang three, the corresponding De Rham group has for generators the (restrictions of) polynomial differential one-forms
$$
ydx, \;xy dx, \;x^2ydx 
$$
which are also generators of the related Brieskorn-Petrov $\C[h]$-module \cite{gavr98}.

Because of the symmetry $(x,y)\rightarrow (\pm x,y)$ the Abelian integrals $I_{2k+1}(h)$ vanish identically, while $I_{2k}$, as well their derivatives can be expressed as linear combinations of $I_0, I_2$, with coefficients in the field $\C(h)$.

\begin{lem}
\noindent\\The integrals $I_i$, $i=0,2$, satisfy the following system of Picard-Fuchs:
\begin{eqnarray*}
I_0(h)&=&\frac{4}{3}hI'_0(h)+\frac{1}{3}I'_2(h) \\
I_2(h)&=&\frac{4}{15}hI'_0(h)+\left(\frac{4}{5}h+\frac{4}{15}\right)I'_2(h)\\
(4h+1)I'_4(h)&=&4hI_0(h)+5I_2(h)\\
4h(4h+1)I''_0(h)&=&-3I_0(h).
\end{eqnarray*}
\end{lem}

The above equations imply the following asymptotic expansions near $h=0$ (they agree with the Picard-Lefshetz formula)
\newpage
\begin{lem}\label{14}
\noindent The integrals $I_i$, $i=0,2$, and $I'_4$ have the following asymptotic expansions in the neighborhood of $h=0$:
\noindent
\begin{eqnarray*}
I_{0}(h)&=&(-h+\frac{3}{8}h^{2}-\frac{35}{64}h^{3}+...)\ln h+\frac{4}{3}+a_{1}h+a_{2}h^{2}+...\\
I_{2}(h)&=&(\frac{1}{2}h^{2}-\frac{5}{8}h^{3}-\frac{315}{256}h^{4}...)\ln h+\frac{16}{15}+4h+b_{2}h^{2}+...\\
I'_{4}(h)&=&(-\frac{3}{2}h^{2}+\frac{35}{8}h^{3}-\frac{471}{256}h^{4}+...)\ln h+\frac{16}{3}+4h+(4a_{1}+5b_{2}-\frac{304}{3})h^{2}+...\\
\end{eqnarray*}
\end{lem}

\section{The monodromy of Abelian integrals}
\label{monodromy}
\begin{figure}
\begin{center}
 \def\svgwidth{9cm}
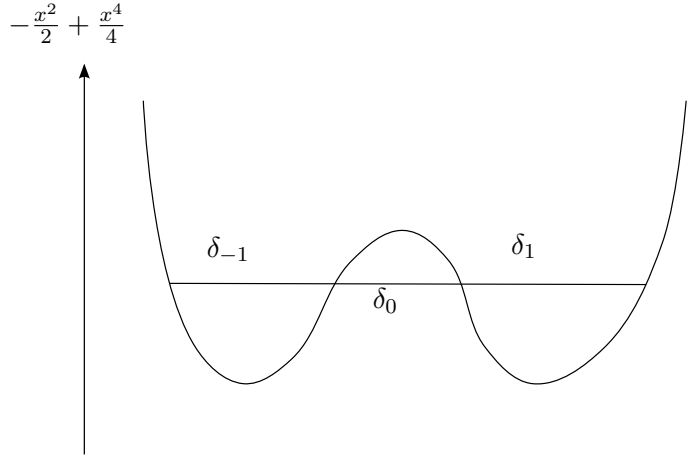
\end{center}
\caption{The vanishing cycles  $\delta_0(h), \delta_1(h), \delta_{-1}(h)$ for $-\frac14<h<0$}
\label{figdelta}
\end{figure}
The Abelian integrals $I(h)$ of the form (\ref{3}) are multivalued functions in $h\in \C$ which become single-valued analytic functions in the complex domain
$$
{\cal D}= \C \setminus [0,-\infty) .
$$
Along the segment $[0,-\infty)$ the integrals have a continuous limit when $h\in {\cal D}$ tends to a point $h_0\in [0,-\infty)$, depending on the sign of the imaginary part of $h$. Namely, if $Im (h)>0$ we denote the corresponding limit by $I^+(h)$, and when $Im (h)>0$ by $I^-(h_0)$. We use a similar notation for the continuous limits of loops $\gamma(h)$ when $h$ tends to the segment $[0,-\infty)$. We have therefore
$$
I^\pm(h)= \int_{\gamma^\pm(h)} \omega
$$
where $\omega$ is a polynomial one-form. The monodromy $I^+(h)-I^-(h)$, $ h\in [0,-\infty)$ depends therefore on the monodromy of $\gamma(h)$ which is expressed by the Picard-Lefscetz formula. Namely, for $h\in {\cal D}$, define the continuous families of closed loops
$$\delta_0(h), \delta_1(h), \delta_{-1}(h)
$$
which vanish at the singular points $(0,0), (0,1), (0,-1)$ when $h$ tends to $0$ or $-1/4$ respectively, and in such a way that
$Im (h) > 0$, see fig.\ref{figdelta}. This defines uniquely  the homology classes of the loops, up to an orientation. From now on we suppose that the loop $\gamma(h)$ for $h>0$ is oriented by the vector field $X_0$, and that the orientation of $\delta_0(h), \delta_1(h), \delta_{-1}(h)$ are chosen in such a way that
$$
\gamma(h) = \delta_0(h) + \delta_1(h)  + \delta_{-1}(h) , \, h \in \cal D .
$$
According to the definition of the vanishing cycles
\begin{equation}
\gamma^+(h) = \delta^+_0(h) + \delta^+_1(h) + \delta^+_{-1}(h) , h\in (-\infty,0] .
\label{gplus}
\end{equation}
and the Picard-Lefschetz formula implies 
\begin{equation}
\gamma^-(h)= - \delta^+_0(h) + \delta^+_1(h) + \delta^+_{-1}(h) , h\in [-1/4,0]
\label{gminus1}
\end{equation}
and
\begin{equation}
\gamma^-(h)= - \delta^+_0(h), h\in (-\infty, -1/4]
\label{gminus2}
\end{equation}
For a further use we note that
\begin{equation}
\delta^-_0(h)= \delta^+_0(h), h\in ( -1/4,+\infty)
\label{d0}
\end{equation}
\begin{equation}
\delta^-_1(h)= \delta^+_1(h), 
\delta^-_{-1}(h)= \delta^+_{-1}(h), 
h\in ( -\infty, 0)
\label{d1}
\end{equation}

\begin{figure}
\begin{center}
 \def\svgwidth{12cm}
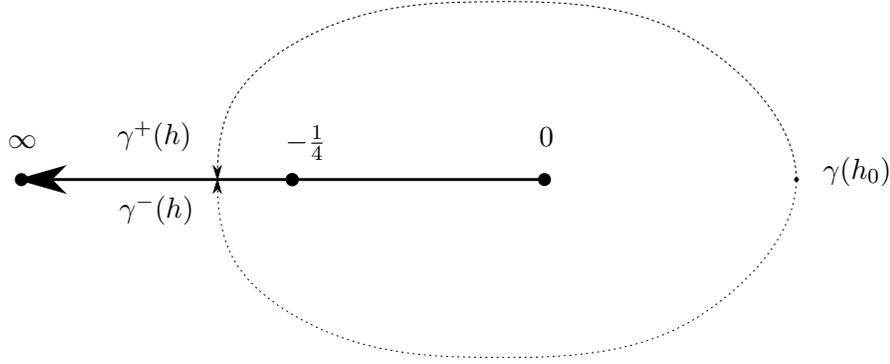
\end{center}
\caption{The analytic continuation of a cycle $\gamma(h)$ in the domain ${\cal D}$}
\label{figgamma}
\end{figure}
\section{The zeros of  the principal part of the first return map in a complex domain}
\label{zeros}
If he  first  Poincar\'e-Pontryagin-Melnikov function $M_1(h)$ is not identically zero, then
\begin{equation}
\label{first}
M_1(h) = \lambda_0 I_0(h) + \lambda_2 I_2, \lambda_i \in \R
\end{equation}
If $M_1=0$ the first non-vanishing Melnikov function $M_k$ can have a more complicated form, and its general form according to \cite{Iliev98}, \cite{Iliev99} is
\begin{equation}
\label{principal}
M(h) = \lambda_0 I_0(h) + \lambda_2 I_2+ \lambda_4 I'_4(h)(h), \lambda_i \in \R
\end{equation}
Following \cite{Li06}, we call 
 the Abelian integral $M(h)$ \emph{the principal part of 
the  first return map} of the system (\ref{f}), associated to the exterior period annulus of $X_0$.
\begin{lem}
The first non-vanishing Poincar\'e-Pontryagin-Melnikov function (\ref{first}) 
has at most two zeros in the complex domain ${\cal D}$.
\label{z1}
\end{lem}
\begin{lem}
The principal part (\ref{principal}) of the  first return map  has at most four zeros in the complex domain ${\cal D}$.
\label{z3}
\end{lem}
\begin{lem}
The Abelian integrals $I_0(h)$ and $I_0'(h)$ do not vanish in ${\cal D}$.
\label{I0}
\end{lem}
\begin{proof}[Proof of Lemma \ref{I0}]
$I_0'(h)$ is a period of the holomorphic one-form $\frac{dx}{y}$ on the elliptic curve $\Gamma_h$, and therefore does not vanish. For real values of $h$ $ I_0'(h)$ represents the period of the orbit $\gamma(h)$ of $d H=0$, while 
$I_0(h)$ equals the area of the interior of $\gamma(h)$.
It is remarkable, that $I_0(h)$ does not vanish in a complex domain too.
Indeed, consider  the analytic function
 $$F(h)= \frac{I_0(h)}{I'_0(h)} , h \in \cal D .$$
 We shall count its zeros in ${\cal D}$ by making use of the argument principle.
\begin{quote}
{\it Let $D\subset\C$ be a relatively compact domain, with a piece-wise smooth boundary.
We suppose, that $f:D\rightarrow\C$ is a continuous function, which is complex-analytic in $D$, except at a finite number of points on the border $\partial D$. We suppose also that $f$ does not vanish on $\partial D$.
Denote by $Z_{D}(f)$ the number of zeros of $f$ in $D$, counted with  multiplicity.
The increment of the argument
\noindent $Var_{\partial D}(argf)$ of $f$ along $\partial D$ oriented counter-clockwise is well defined 
and equals the winding number of the curve $f(\partial D)\subset \C$ about the origin, divided by $2\pi$. The argument principle states then that }
\begin{eqnarray}\label{10}
2\pi Z_{D}(f)&=&Var_{\partial D}(argf)
\end{eqnarray}
\end{quote}
Apply now the argument principle to the function $F$ in the intersection of a big disc with a radius $R$ and the complex domain ${\cal D}$. Along the  circle of radius $R$, for $R$ sufficiently big, the decrease of the argument of $F$ is close to $2\pi$, while along the branch cut $(-\infty, 0)$ we have
$$
2 \sqrt{-1} Im(F(h) = F^+(h)-F^-(h)= 
\frac{I_{0}(h)}{I'_{0}(h)}-\frac{\overline{I_{0}(h)}}{\overline{I'_{0}(h)}}$$
$$
=
\frac{\oint_{\gamma^+}ydx}{\oint_{\gamma^+}\frac{dx}{y}}-\frac{\oint_{\gamma^{-}}ydx}{\oint_{\gamma^{-}}\frac{dx}{y}}=
\frac{W(h)}{|\oint_{\gamma^+}\frac{dx}{y}|^2} .
$$
where
$$
W(h)=det\left(\begin{array}{cc} \oint_{\gamma^+}ydx & \oint_{\gamma^+}\frac{dx}{y}\\ & \\ \oint_{\gamma^-}ydx & \oint_{\gamma^-}\frac{dx}{y} \end{array}\right) .
$$
According to section \ref{monodromy}, the function has two different determinations along $(-\infty,-1/4)$ and $(-1/4,0)$, both of which have no monodromy, and hence are rational in $h$. In fact, (\ref{area}) implies that $W(h)$ is a non-zero constant. If $W(h)=c$ in $(-\infty,-1/4)$, then it equals $2c$ in $(-1/4,0)$. Therefore along the branch cut
the argument of $F^+$ or $F^-$ increases by at most $\pi$. 
Summing up the above information, we conclude that $F$ has no zeros in ${\cal D}$.
\end{proof}
\begin{proof}[Proof of Lemma \ref{z1}]
We denote
$$
F(h) = \frac{M_1(h}{I_0(h)} =  \lambda_2 \frac{I_2(h}{I_0(h)} + \lambda_0, h\in \cal D 
$$
and apply, as in the proof of Lemma \ref{I0}, the argument principle to $F$. Along a big circle the increase of the  argument of $F$ is close to $\pi$. Along the branch cut $(-\infty,0]$ we have  
$$
2 \sqrt{-1} Im(F(h)) = F^+(h)-F^-(h)=  \lambda_2 \frac{W(h)}{|I_0(h)|^2}
$$
where 
$$
W(h)=
\det\left(\begin{array}{cc} \oint_{\gamma^+}yx^2dx & \oint_{\gamma^+} ydx \\ & \\ \oint_{\gamma^-}yx^2dx & \oint_{\gamma^-}ydx \end{array}\right) = c h (4h+1), c= const.\neq 0 .
$$
Therefore the imaginary part of $F(h)$ along the branch cut $(-\infty,0)$ vanishes at most once, at $-1/4$.
Summing up the above information, we get that $F$ has at most two zeros in the complex domain ${\cal D}$.
\end{proof}
\begin{proof}[Proof of Lemma \ref{z3}]
We denote
$$
F(h) = (4h+1) \frac{M(h}{I_0(h)} , h\in \cal D 
$$
and apply, as in the proof of Lemma \ref{I0}, the argument principle to $F$.  By making use of (\ref{I4}) we have
\begin{equation}
F(h)= \alpha(h) \frac{I_2(h}{I_0(h)} + \beta(h)
\label{fh}
\end{equation}
where 
\begin{equation}
\alpha(h)= (4h+1) \lambda_2 + 5 \lambda_4, \; \beta(h)= (4h+1) \lambda_0+ 4 h \lambda_4 .
\label{albe}
\end{equation}
Along a big circle the increase of the  argument of $F$ is close to $3\pi$. Along the branch cut $(-\infty,0]$ we have as before
$$
2 \sqrt{-1} Im(F(h)) = F^+(h)-F^-(h)=  \alpha(h) \frac{W(h)}{|I_0(h)|^2}
$$
where 
$$
W(h)=
\det\left(\begin{array}{cc} \oint_{\gamma^+}yx^2dx & \oint_{\gamma^+} ydx \\ & \\ \oint_{\gamma^-}yx^2dx & \oint_{\gamma^-}ydx \end{array}\right) = c h (4h+1), c= const.\neq 0 .
$$
Therefore the imaginary part of $F(h)$ along the branch cut $(-\infty,0)$ vanishes at most twice, at $-1/4$ and at the root of $\alpha(h)$.
Summing up the above information, we get that $F$ has at most four zeros in the complex domain ${\cal D}$.
\end{proof}

\section{The bifurcation diagram of the zeros of the Abelian integrals in a complex domain}
\label{section5}
\subsection{ The first Melnikov function $M_1$}
Let $Z(M_1)$ be the number of the zeros of $M_1(h)$ in the domain ${\cal D}$, counted with multiplicity. It is a function of $[\lambda_0:\lambda_2]$ seen as a point on the projective circle $S^1=\R\mathbb P^1$. The bifurcation set ${\mathbf B}$ of $Z(M_1)$ is the set of points
$[\lambda_0:\lambda_2]\in \R\mathbb P^1$ at which $Z(M_1)$ is not a locally constant function. It follows that if $\lambda=[\lambda_0:\lambda_2]$ is a bifurcation point, then near   $\lambda$ a zero of $M_1(h)$ bifurcates from the border of the domain ${\cal D} \subset \C{\mathbb P}^1$, see \cite[Definition 2]{Gav01}. 
Therefore
$$
{\bf B} = P_0\cup P_{-1/4}\cup P_\infty \cup \Delta
$$ 
where $P_0, P_{-1/4},  P_\infty\in S^1$ are the sets of parameter values $\lambda$, corresponding to bifurcations of zeros from $h=0$, $h=-1/4$ and $h=\infty$ respectively. Finally, $\Delta $ is the set corresponding to bifurcations from the branch cut $(-\infty,0)$. The results of the preceding section imply $\Delta= \emptyset$ while
$$
P_0= \{ [\lambda_0 :\lambda_2] : \lambda_0 I_0(0)+\lambda_2 I_2(0) = 0 \}, P_{-1/4}= \{ [\lambda_0 :\lambda_2] : \lambda_0 I_0(-\frac14)+\lambda_2 I_2(-\frac14) = 0 \}
$$
and
$$
P_\infty =  \{ [\lambda_0 :\lambda_2] : \lambda_2=0 \} .
$$
A local analysis shows that when the parameter $[\lambda_0:\lambda_2]$ crosses $P_0$ or $P_\infty$, then a simple zero bifurcates from $0$ or $\infty$. Similarly, two complex conjugate zeros bifurcate from $h=-1/4$ when 
$[\lambda_0:\lambda_2]$ crosses $P_{-1/4}$. This combined with Lemma \ref{z1} implies
\begin{coro}
The bifurcation diagram of $Z(M_1)$ together with the corresponding number of zeros  of $M_1$ are shown on fig.\ref{diagram2}
\end{coro}
\begin{figure}
\begin{center}
 \def\svgwidth{9cm}
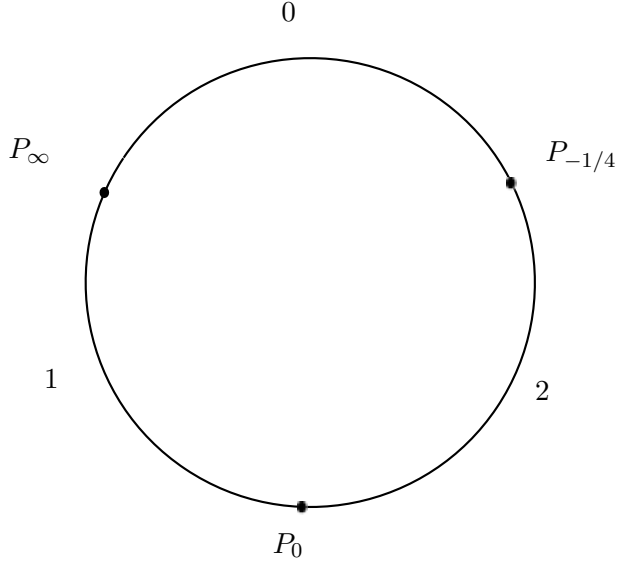
\end{center}
\caption{Bifurcation diagram of the zeros of the first Melnikov function $M_1$ in the complex domain ${\cal D}$.}
\label{diagram2}
\end{figure}
\subsection{The principal part $M$ of the first return map.}
\label{bifurcation}
\begin{figure}
\begin{center}
 \def\svgwidth{8cm}
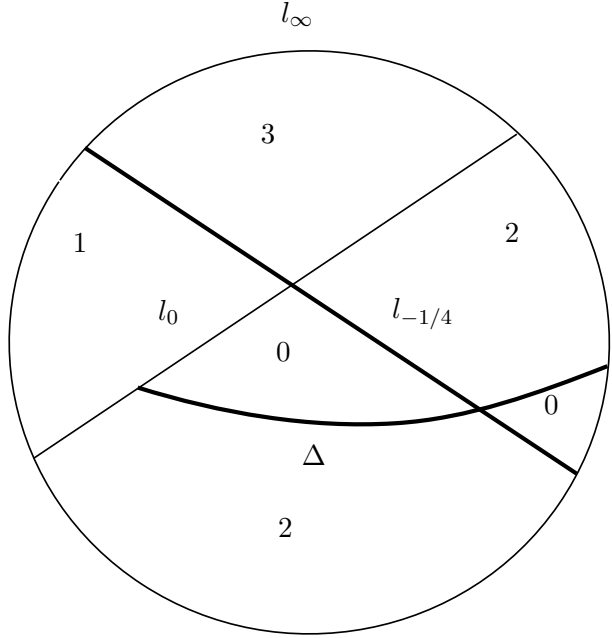
\end{center}
\caption{Bifurcation diagram of the zeros of  $M$ in the complex domain ${\cal D}$.}
\label{bset}
\end{figure}
Let $Z(M)$ be the number of the zeros of $M(h)$ in the domain ${\cal D}$, counted with multiplicity. It is a function of $[\lambda_0:\lambda_2: \lambda_4]$ seen as a point on the projective sphere $\R\mathbb P^2$. The bifurcation set ${\mathbf B}$ of $Z(M)$ is the set of points
$[\lambda_0:\lambda_2: \lambda_4]\in \R\mathbb P^2$ at which $Z(M)$ is not a locally constant function. It follows that if $\lambda=[\lambda_0:\lambda_2: \lambda_4]$ is a bifurcation point, then near such a $\lambda$ a zero of $M(h)$ bifurcates from the border of the domain ${\cal D} \subset \C{\mathbb P}^1$, see \cite[Definition 2]{Gav01}. 
Therefore
$$
{\bf B} = l_0\cup l_{-1/4}\cup l_\infty \cup \Delta
$$ 
where $l_0, l_{-1/4},  l_\infty$ are the sets of parameter values $\lambda$, corresponding to bifurcations of zeros from $h=0$, $h=-1/4$ and $h=\infty$ respectively. Finally, $\Delta $ is the set corresponding to bifurcations from the branch cut $(-\infty,0)$. We are going to describe these sets explicitly.

As $I'_4(0) \neq 0$, then 
\begin{equation}
l_0= \{\lambda\in \R\mathbb P^2 : \lambda_0 I_0(0) + \lambda_2 I_2(0)+\lambda_4 I'_4(0) = 0 \}.
\label{l0}
\end{equation}
Similarly, 
\begin{equation}
l_\infty = \{ \lambda\in \R\mathbb P^2 :  \lambda_2 = 0 \} .
\label{linfty}
\end{equation}
A local analysis shows that $I'_4(h) \sim const. \times \log(4h+1)$ near $-1/4$ which implies
\begin{equation}
l_{-1/4} = \{ \lambda\in \R\mathbb P^2 : \lambda_4 = 0 \} .
\label{l14}
\end{equation}
Finally, to compute $\Delta$ we suppose that for some $h\in (-\infty,-1/4) \cup (-1/4,0)$, $M(h)=\overline{M(h)}=0$. The latter implies $Im( I(h))=0$, and hence $\alpha(h)=0$ and $\beta(h)=0$, see (\ref{fh}). The condition, that the polynomials $\alpha(h), \beta(h)$ have a common real root imply that either $\lambda_4=0$ in which case the root is $h=-1/4$, or 
\begin{equation}
5(\lambda_0+\lambda_4) + \lambda_2=0
\label{delta}
\end{equation}
in which case
$$
(4h+1) (\lambda_0I_2(h)+\lambda_2I_2(h)+\lambda_4I'_4(h))=
(4h(\lambda_0+\lambda_4) + \lambda_0) (I_0(h)-5I_2(h))
$$
see (\ref{albe}). The Abelian integral $I_0(h)-5I_2(h)$ vanishes at $h=-1/4$ and corresponds therefore to the point $P_{-1/4}$ on fig.\ref{diagram2}. In particular it has no zeros in the domain ${\cal D}$. Thus, in the case when the root of $4h(\lambda_0+\lambda_4) + \lambda_0$ belongs to ${\cal D}$, the Abelian integral $M$ has exactly one zero in ${\cal D}$, otherwise it has complex conjugate zeros on $(-\infty,0)$. This implies from one hand that $\Delta$ is the segment of the line (\ref{delta}), connecting $l_\infty$ and $l_0$ as on fig.\ref{bset}. Thus also implies that in one of the connected components of  $ \R\mathbb P^2\setminus \mathbb B$ the fundtion $M$ has exactly one zero, as shown on  fig.\ref{bset}.
To determine the number od the zeros of $M$ in the remaining connected components of the complement to the bifurcation set in $ \R\mathbb P^2$ we note that
\begin{itemize}
\item when crossing $\Delta$ or $l_{-1/4}$ (in bold on the figure) two zeros are added or subtracted 
\item  when crossing $l_0$ or $l_{\infty}$  one simple zero is added or subtracted 
\item the total number of zeros of $M$ is not bigger than three
\end{itemize}
The above considerations, combined with Lemma \ref{z3}, determine uniquely the number of the zeros of $M$ in each connected component. This is summarized in the following
\begin{thm}
\label{main}
 The bifurcation set ${\bf B}\subset {\mathbb P}^2$ of the zeros $Z(M)$ of the  principal part of the return map,  in the complex domain ${\cal D}\subset {\mathbb C}$ is the union of the projective lines $l_0,  l_{-1/4},  l_\infty $ and the segment $\Delta$ connecting $l_0$ to $l_\infty$. Their mutual position, together with the corresponding number of zeros  of $M$ are shown on fig.\ref{bset}.
\end{thm}  

The bound for the number of the zeros in the above Corollary  in the complex domain ${\cal D}$ is three, 
which, according to Theorem \ref{cyclicity}, 
is not optimal on the real interval $(0,\infty)$. It seems impossible to deduce Theorem \ref{cyclicity} from Theorem \ref{main} by making use of complex methods only.

\end{document}